\documentclass[a4paper]{amsart}

\pdfoutput=1

\usepackage[english]{babel}

\usepackage{amssymb, amsthm}
\usepackage{mathtools}

\usepackage{microtype}
\usepackage[T1]{fontenc}

\usepackage[a4paper, margin=2.5cm]{geometry}

\usepackage{tikz}

\usepackage[colorlinks, bookmarksdepth=3]{hyperref}

\theoremstyle{plain}
\newtheorem{thm}{Theorem}[section]
\newtheorem{prop}[thm]{Proposition}
\newtheorem{cor}[thm]{Corollary}
\newtheorem{lem}[thm]{Lemma}
\newtheorem{conj}[thm]{Conjecture}

\theoremstyle{definition}
\newtheorem{defn}[thm]{Definition}

\newtheorem{rem}[thm]{Remark}
\newtheorem{quest}{Question}

\DeclareMathOperator{\conv}{conv}
\DeclareMathOperator{\GL}{GL}

\newcommand{\flt}{{\mathrm{Flt}}}
\newcommand{\fltR}{\flt^\R_d}

\newcommand{\lspan}{{\mathrm{span}}}
\newcommand{\width}{{\mathrm{width}}}
\newcommand{\CSR}{{\mathrm{CSR}}}
\newcommand{\SR}{{\mathrm{SR}}}

\newcommand{\C}{{\mathbb{C}}}
\newcommand{\N}{{\mathbb{N}}}

\newcommand{\R}{{\mathbb{R}}}
\newcommand{\Z}{{\mathbb{Z}}}

\newcommand{\rleft}{\mathopen{}\mathclose\bgroup\left}
\newcommand{\rright}{\aftergroup\egroup\right}
\newcommand{\vect}[1]{{\boldsymbol{\mathbf{#1}}}}

\newcommand{\with}{\colon}

\newcommand{\setcond}[2]{\left\{ #1\with #2\right\}}

\newcommand{\va}{{\vect{a}}}
\newcommand{\vb}{{\vect{b}}}
\newcommand{\ve}{{\vect{e}}}
\newcommand{\vu}{{\vect{u}}}
\newcommand{\vp}{{\vect{p}}}
\newcommand{\vv}{{\vect{v}}}
\newcommand{\vw}{{\vect{w}}}
\newcommand{\vx}{{\vect{x}}}
\newcommand{\vo}{\vect{0}}

\newcommand{\pro}[2]{#1(#2)}
\newcommand{\probracket}[2]{(#1)(#2)}

\begin{document}
\selectlanguage{english}


\title{Generalized flatness constants, spanning lattice polytopes, and the Gromov width}

\author[G.\,Averkov]{Gennadiy Averkov}
\address[G.\,Averkov]{BTU Cottbus-Senftenberg\\Platz der Deutschen Einheit 1\\03046 Cottbus\\Germany}
\email{averkov@b-tu.de}

\author[J.\,Hofscheier]{Johannes Hofscheier}
\address[J.\,Hofscheier]{School of Mathematical Sciences\\University of Nottingham\\Nottingham, NG7 2RD\\UK}
\email{johannes.hofscheier@nottingham.ac.uk}

\author[B.\,Nill]{Benjamin Nill}
\address[B.\,Nill]{Faculty of Mathematics\\Otto-von-Guericke Universit\"at Magdeburg\\Postschlie{\ss}fach 4120\\ 39106 Magdeburg\\Germany.}
\email{benjamin.nill@ovgu.de}


\begin{abstract}
  In this paper we motivate some new directions of research regarding the lattice width of convex bodies. We show that convex bodies of sufficiently large width contain a unimodular copy of a standard simplex. This implies that every lattice polytope contains a minimal generating set of the affine lattice spanned by its lattice points such that the number of generators is bounded by a constant which only depends on the dimension. We also discuss relations to recent results on spanning lattice polytopes and how our results could be viewed as the beginning of the study of generalized flatness constants. Regarding symplectic geometry, we point out how the lattice width of a Delzant polytope is related to upper and lower bounds on the Gromov width of its associated symplectic toric manifold. Throughout, we include several open questions.
\end{abstract}

\maketitle{}


\section{Summary}
\label{sec:intro}

In this paper we discuss open problems related to the lattice width (and generalizations) motivated by questions on lattice polytopes and symplectic manifolds. Our main contributions are:

\begin{enumerate}
\item Convex bodies of ``large'' width contain unimodular copies of the standard simplex (Theorem~\ref{thm:flatness}).
\item The existence of ``small'' subsets of lattice points in a lattice polytope, i.e., their size is bounded by a constant that only depends on the dimension, which generate the lattice spanned by all lattice points in the given polytope (Corollary~\ref{cor:min}).
\item Examples of three-dimensional lattice polytopes of arbitrarily large width that are not very ample (Proposition~\ref{prop:example}).
\item Evidence that the lattice width is an upper bound on the Gromov width (Proposition~\ref{equality-widths}).
\item A lower bound on the Gromov width in terms of the lattice width (Corollary~\ref{cor:gw}).
\item A characterization of Gromov width $1$ for Delzant lattice polygons (Proposition~\ref{gw1}).
\end{enumerate}

\subsection*{Organization of the paper}
Section~\ref{sec:convex} discusses our results in the context of the geometry of numbers, Section~\ref{sec:sympl} relates them to symplectic geometry. Finally, Section~\ref{sec:flatproof} gives the proof of Theorem~\ref{thm:flatness}, while Section~\ref{sec:exampleproof} contains the proof of Proposition~\ref{prop:example}.

\subsection*{Acknowledgments}
We are particularly grateful to Milena Pabiniak for introducing us to combinatorial problems regarding the Gromov width and patiently answering our questions. We would also like to thank Gabriele Balletti for useful discussions and computations regarding non-spanning lattice polytopes. The first and last authors are PIs in the Research Training Group Mathematical Complexity Reduction funded by the Deutsche Forschungsgemeinschaft (DFG, German Research Foundation) - 314838170, GRK 2297 MathCoRe. The last author is also partially supported by the Vetenskapsr\aa det grant NT:2014-3991 (as an affiliated researcher with Stockholm University).


\section{The main results from the viewpoint of the geometry of numbers}
\label{sec:convex}

\subsection{Convex bodies of large lattice width}

In order to formulate the main result of this section, let us fix some notation. Recall that a non-empty compact and convex subset $K \subseteq \R^d$ is called a \emph{convex body}. The \emph{width} of a convex body $K \subseteq \R^d$ with respect to a non-zero linear functional $\vu \in \rleft( \R^d \rright)^*=\mathrm{Hom}(\R^d, \R)$ is given as
\[
  \width_\vu(K) \coloneqq \max_{x,y \in K} |\vu(x) - \vu(y)| \text{,}
\]
and the \emph{(lattice) width} of $K$ is defined as
\[
  \width(K) \coloneqq \min_{\vu \in \rleft( \Z^d \rright)^* \setminus \rleft\{
  \mathbf{0} \rright\}} \width_\vu (K) \text{,}
\]
where $(\Z^d)^* = \mathrm{Hom}(\Z^d, \Z)$ denotes the dual lattice.

Here is the classical definition of the {\em flatness constant} in dimension $d$:
\[
  \flt(d) \coloneqq \sup \rleft\{ \width(K) \colon K \subseteq \R^d \; \text{convex body}, K \cap \Z^d = \emptyset \rright\} \text{.}
\]
It is known that $\flt(d) \le O(d^{\frac{3}{2}})$ by \cite{Flatnesspaper}. An explicit upper bound of order $O(d^{\frac{5}{2}})$ 
is given by $\flt(d) \le \sqrt{\frac{(d+1)(2d+1)}{6}} d^{\frac{3}{2}}$ \cite[Theorem~(7.4), (8.3)]{Barvinokbook}. Clearly, $\flt(1) = 1$ while in higher dimensions it is known that $\flt(2)= 1+\frac{2}{\sqrt{3}}$ by \cite{Hurkens} and $\flt(3) \le 4.244$ by \cite{flat3d}. The most recent result known to the authors is \cite[Conjecture 1.1]{FlatDimThree} where it is conjectured that $\flt(3) = 2 + \sqrt{2}$.

Let us recall some more notation. The set $\GL(d,\Z)$ consists of the $d \times d$-matrices with integer coefficients and determinant $\pm 1$. We will call a map $T \colon \R^d \to \R^d,\; \vx \mapsto A \vx + \vb$ with $A \in \GL(d,\Z)$ and $\vb \in \Z^d$ an (affine) {\em unimodular transformation}, respectively, for $\vb \in \R^d$ an (affine) {\em $\R$-unimodular transformation}. For $X \subseteq \R^d$, we call $T(X)$ a \emph{unimodular copy}, respectively, an \emph{$\R$-unimodular copy} of $X$.

We define the $d$-dimensional \emph{standard simplex} as 
\[
  \Delta_d \coloneqq \conv(\vect{0}, \ve_1, \ldots, \ve_d),
\]
where $\ve_1, \ldots, \ve_d$ is the standard basis of $\R^d$. Here is our main observation.

\begin{thm}
  \label{thm:flatness}
  Let $K \subseteq \R^d$ be a convex body. 
  \begin{enumerate}
  \item If $\width(K) \ge 2 \flt(d) d$, then $K$ contains a unimodular copy of the standard simplex.
  \item If $\width(K) \ge \flt(d) d$, then $K$ contains an $\R$-unimodular copy of the standard simplex.
  \end{enumerate}
\end{thm}

The proof will be given in Section~\ref{sec:flatproof}. Let us note the following version of Theorem~\ref{thm:flatness}(2) (apply it to $\frac{\flt(d) d}{\width(K)} \cdot K$).

\begin{cor}
  \label{cor:flatness}
  Any convex body $K \subseteq \R^d$ contains an $\R$-unimodular copy of $\frac{\width(K)}{\flt(d) d} \cdot \Delta_d$.
\end{cor}

\subsection{Lattice-generating subsets of bounded size}  

To a lattice polytope $P$, one associates the semigroup of lattice points in $\Z^{d+1}$ in the cone over $P \times \{1\}$. In this case, there is a unique minimal set of generators of this semigroup,  called its {\em Hilbert basis} (we refer to the book \cite{BGbook} as a pointer to the extensive literature on this topic). Clearly, the size of the Hilbert basis does not accept a bound which only depends upon the dimension. However, if one replaces this semigroup by the subgroup of $\Z^{d+1}$ generated by the lattice points of $P$, then the size of a generating subset of the lattice points of $P$ is bounded by a function in the dimension:
\begin{cor}
  \label{cor:min}
  For any $d \ge 0$, there is a constant $C(d)$ with the following property. Given a $d$-dimensional lattice polytope $P \subseteq \R^d$, there exists a subset $A \subseteq P \cap \Z^d$ with $|A| < C(d)$ that generates the affine lattice spanned by $P \cap \Z^d$.
\end{cor}
\begin{proof}
  We set $C(d) \coloneqq \prod_{k=1}^d (2 \flt(k) \cdot k + 1)$ for $d \ge 1$ and $C(0) \coloneqq 2$. Since $\flt(d)\ge1$ for $d \ge1$ (consider the unit cube),  we have $C(d) > d + 1$.
    
  We do induction on $d \ge 0$. The base cases $d=0, 1$ straightforwardly follow (clearly $C(0) = 2$ and $C(1) = 3$ work). So, let $d \ge 2$. If $\width(P) \ge 2 \flt(d) d$, then by Theorem~\ref{thm:flatness}(1) we can choose $d+1$ lattice points from $P$ that span the affine lattice generated by $P \cap \Z^d$ (note that $d+1 < C(d)$). Otherwise, we may assume that $P \cap \Z^d \subseteq \Z^{d-1} \times \{0, \ldots, w\}$ with $w \coloneqq \width(P) < 2 \flt(d) d$. Consider the maximal subset $J \subseteq \{0, \ldots, w\}$ such that $P_j \coloneqq \conv(\setcond{(x_1,\ldots,x_d) \in P \cap \Z^d}{x_d=j}) \subseteq \R^{d-1} \times \{j\}$ define non-empty lattice polytopes for every $j \in J$ (\emph{non-empty} is meant in the set theoretic sense). Then by our induction hypothesis, for any $j \in J$ there exist subsets $A_j \subseteq P_j \cap \Z^d$ with $|A_j| < C(\dim(P_j))$ that generate the affine lattices spanned by $P_j \cap \Z^d$. Clearly, $A \coloneqq \bigcup_{j \in J} A_j$ is a generating set of the affine lattice spanned by $P \cap \Z^d$. We observe $|A| < (w+1) C(d-1) < (2 \flt(d) d +1) C(d-1) = C(d)$.
\end{proof}

In \cite{Seboe-Hilbert, Bruns-Cara}, the {\em Carath\'eodory rank} of a cone is defined as the minimal $n \in \N$ such that every lattice point in the semigroup of lattice points in the cone can be written as a nonnegative integral combination of at most $n$ Hilbert basis elements. In \cite{GubeladzeNormal}, Gubeladze extends the Carath\'eodory rank to polytopes by considering the cone over the polytope and then reducing the study to the cone case. Now, Corollary~\ref{cor:min} motivates the following related definition (where we use the linear instead of the affine setting in order to stress the analogy).

\begin{defn}
  \label{cara:def}
  Let $P \subseteq \R^d$ be a $d$-dimensional lattice polytope, and $\Lambda \subseteq \Z^{d+1}$ the lattice spanned by the lattice points in $P \times \{1\}$.
  \begin{itemize}
  \item The \emph{Carath\'eodory spanning rank} $\CSR(P)$ of $P$ is defined as the minimal $n \in \N$ such that every lattice point in $\Lambda$ can be written as an integer combination of at most $n$ lattice points in $P \times \{1\}$.
  \item The \emph{spanning rank} $\SR(P)$ of $P$ is defined as the minimal size of a lattice generating set of $\Lambda$ contained in $P \times \{1\}$.
  \end{itemize}
  Similar to the flatness constant $\flt(d)$, we introduce the \emph{(maximal) spanning rank in dimension $d$}, $\SR(d)$, as the maximal spanning rank $\SR(P)$ of $d$-dimensional lattice polytopes.
\end{defn}

We remark that $\CSR(P) \le \SR(P) \le (d+1) \CSR(P)$ (for the second inequality consider a lattice basis of $\Lambda$). Clearly, $\SR(1)=2$ and $\SR(2)=3$ where the second equality follows by the fact that empty lattice polygons are unimodular. In dimension $d=3$, it can be deduced from \cite[Theorem~1.7]{nonspanning} that $\SR(3)=5$.

\begin{rem}
  Note that from the proof of Corollary \ref{cor:min} it directly follows that for $d\ge1$ we have $\SR(d) \le \prod_{k=1}^d(2\flt(k)\cdot k+1)$. Recall from above that $\flt(k) \le O(k^{3/2})$, i.e., there is a positive constant $M$ such that $\flt(k) \le M k^{3/2}$. We may assume $M\ge1$. Then, we get the  upper bound
  \[
    \SR(d) \le \prod_{k=1}^d(2\flt(k)\cdot k+1) \le \prod_{k=1}^d (2Mk^{\tfrac{5}{2}}+1) \le \prod_{k=1}^d 4Mk^{\tfrac{5}{2}} \le (4M)^d (d!)^{\tfrac{5}{2}} \text{.}
  \]
\end{rem}

\begin{quest}
  What is the asymptotical order of $\SR(d)$?
\end{quest}

\begin{quest}
  Is there an efficient algorithm to find $\SR(P)$ (respectively, $\CSR(P)$) for a given lattice polytope $P$?
\end{quest}

\subsection{Bounding the width of non-spanning lattice polytopes}

A lattice polytope $P \subseteq \R^d$ is called {\em spanning} if every lattice point in $\Z^d$ is an affine integral combination of lattice points in $P$. In the notation of Definition~\ref{cara:def}, this is equivalent to $\Lambda=\Z^{d+1}$. Spanning lattice polytopes are in the focus of current research in classifications and Ehrhart theory of lattice polytopes as they form a large class of lattice polytopes that have nice Ehrhart-theoretic behavior \cite{spanning1, spanning2, nonspanning}. On the other side, non-spanning lattice polytopes are quite exceptional, for instance, they include the much-studied class of empty lattice simplices. In \cite{nonspanning} building upon \cite{Monica}, a complete classification of all non-spanning three-dimensional lattice polytopes was achieved. In particular, it follows from their main result (\cite[Theorem~1.3]{nonspanning}) that the width of a non-spanning lattice polytope of dimension $d=3$ is at most $3$ (compare this with the maximal width $1$ of empty lattice tetrahedra). Here, it follows from Theorem~\ref{thm:flatness}(1) that such an upper bound exists in any dimension.

\begin{cor}
  \label{cor-span}
  Any $d$-dimensional lattice polytope $P \subseteq \R^d$ with $\width(P) \ge 2 \flt(d) d$ is spanning.
\end{cor}

\begin{rem}
  Having a unimodular copy of the standard simplex is in general stronger than being spanning, however, at least in lower dimensions surprisingly not by much. Theorem~1.7 in \cite{nonspanning} shows that there are only two spanning lattice $3$-polytopes that do not contain a unimodular copy of the standard simplex.
\end{rem}

Note that for $d=3$, Corollary~\ref{cor-span} implies that the width of a non-spanning lattice polytope is at most 25 by using the bound from \cite{flat3d}. Compare this to the sharp bound of $3$. It would be interesting to know the correct asymptotical order.

\begin{quest}
  What is the maximal width of a non-spanning lattice polytope in dimension $d$?
\end{quest}

We remark that the more classical situation of large dilations (instead of large width) is much easier and completely understood.

\begin{prop}
  Let $P \subseteq \R^d$ be a $d$-dimensional lattice polytope. Then $k P$ is spanning for $k \ge \lfloor \frac{d+1}{2}\rfloor$.
\end{prop}
\begin{proof}
  We may assume that $P$ is a lattice simplex (otherwise triangulate $P$ and use one of the simplices in the triangulation). Consider the closed parallelepiped $\Pi$ spanned by $P \times \{1\} \subseteq \R^{d+1}$ (see e.g. \cite{BeckBook}). The lattice $\Z^{d+1}$ is spanned by all the lattice points in $\Pi$ which includes the vertices $\vv_0, \ldots, \vv_d$ of $P \times \{1\}$. As $\Pi$ has the symmetry $\vx \mapsto \vv_0 + \ldots + \vv_d - \vx$, we see that $\Z^{d+1}$ is already spanned by all the lattice points in $\Pi$ with last coordinate $\le \frac{d+1}{2}$. From this the statement follows.
\end{proof}

The previous bound is sharp in any odd dimension $d \ge 3$: Consider for instance the unique empty lattice simplex of normalized volume $2$ (e.g., see \cite{smallvolume}).

\subsection{Further properties of lattice polytopes of large width?}

It is tempting to conjecture that lattice polytopes of large width satisfy even stronger properties than spanning such as \emph{IDP} (integrally-closed) or \emph{very ample} (we refer to \cite{BGbook} for the precise definitions). However, this is not true.

\begin{prop}
  \label{prop:example}
  For any dimension $d \ge 3$ and any integer $k \ge 3$ there exists a $d$-dimensional lattice polytope $P \subseteq \R^d$ of width $k$ such that for any integer $t \ge 2$ there is a lattice point in $t P$ that is not the sum of $t$ lattice points in $P$.
\end{prop}

The proof will be given in Section~\ref{sec:exampleproof}. Recall Gubeladze's work \cite{long-edges} where he proved that lattice polytopes with sufficiently long edges are IDP. The previous observation shows that this cannot be generalized to lattice polytopes of sufficiently large width. From Gubeladze's result which was motivated by Oda's conjecture (which asks whether Delzant lattice polytopes are IDP), the following weaker question emerges (recall that a polytope is \emph{Delzant} if all its normal cones are unimodular):

\begin{quest}
  Is there a constant $N(d)$ such that any $d$-dimensional Delzant lattice polytope with width at least $N(d)$ is IDP?
\end{quest}

Another main conjecture in this field is the question whether IDP implies unimodality of the $h^*$-vector (we refer to \cite{Schepers}). In particular, it is expected that lattice polytopes with sufficiently long edges have a unimodal $h^*$-vector. This motivates the following question.

\begin{quest}
  Is there an infinite family of $d$-dimensional lattice polytopes with arbitrarily large widths that all have a non-unimodal $h^*$-vector?
\end{quest}

\subsection{Generalized flatness constants}
 
Finally, we would like to rephrase Theorem~\ref{thm:flatness} in terms of generalized flatness constants, as this seems to us a promising unifying approach to several of the above questions.

\begin{defn}
  \label{def:flatness}
  For a bounded subset $X \subseteq \R^d$, we define the \emph{flatness constant} with respect to $X$ by
  \[
    \flt_d(X) \coloneqq \sup \rleft\{ \width(K) \with K \subseteq \R^d \;
    \text{convex body}, K  \text{ does not contain a unimodular copy of} \; X
    \rright\} \text{,}
  \]
  and the \emph{$\R$-flatness constant} with respect to $X$ by
  \[
    \fltR(X) \coloneqq \sup \rleft\{ \width(K) \with K \subseteq \R^d \;
    \text{convex body}, K  \text{ does not contain an $\R$-unimodular copy of} \; X
    \rright\} \text{.}
  \]
\end{defn}

For $X = \rleft\{ \vect{0} \rright\} \subseteq \R^d$, we recover the usual flatness constant, i.e.,
\[
  \flt_d(\rleft\{ \vect{0} \rright\}) = \flt(d) \text{.}
\]
We remark that $\fltR(X) \le \flt_d(X)$, and both generalized flatness constants are monotone with respect to inclusion. $\flt_d$ is invariant under unimodular transformations while $\fltR$ is invariant under $\R$-unimodular transformations. Moreover, it is straightforward to show that $\flt_d^\R(nX) = n\flt_d^\R(X)$ for any positive real number $n$ while the analogous statement for $\flt_d(\cdot)$ is a priori not clear.

Theorem~\ref{thm:flatness} implies that these generalized flatness constants are real numbers.

\begin{cor}
  Let $X \subseteq \R^d$ be a bounded subsets which fits in a unimodular copy of $n \cdot \Delta_d$. Then $\flt_d(X) \le 2nd \cdot \flt(d)$ and $\fltR(X) \le nd \cdot \flt(d)$.
\end{cor}

Motivated by its applications for lattice polytopes, we ask:

\begin{quest}
  What is the order of $\flt_d(n \cdot \Delta_d)$?
\end{quest}

In order to find lower bounds on the Gromov width the following question is of interest (see Theorem~\ref{lowerbd} and Theorem~\ref{thm:cross}). For this let us define the $d$-dimensional {\em standard crosspolytope} as 
\[
  \Diamond_d \coloneqq \conv(\pm \ve_1, \ldots, \pm \ve_d) \text{.}
\]

\begin{quest}
  \label{quest:flat-consts}
  What is the order of $\fltR(n \cdot \Delta_d)$, respectively of $\fltR(n \cdot \Diamond_d)$?
\end{quest}

\begin{rem}
  Note that the constant in Question \ref{quest:flat-consts} does't play a role (recall $\fltR(n \cdot X) = n \cdot \fltR(X)$).
\end{rem}

Let us mention an interesting relationship between $\flt_d(n\cdot X)$ and $\fltR(X)$ as $n$ goes to infinity.
The following lemma will be crucial in this explanation:

\begin{lem}
  For any $X \subseteq \R^d$, we have $\flt_d(X) \le \fltR(X+[0,1]^d)$.
\end{lem}
\begin{proof}
  It suffices to show that any convex body $K \subseteq \R^d$ which does not contain a unimodular copy of $X$, cannot contain an $\R$-unimodular copy of $X + [0,1]^d$.
  We show the contraposition, i.e., any convex body $K \subseteq \R^d$ which contains an $\R$-unimodular copy of $X + [0,1]^d$ does also contain a unimodular copy of $X$.
  Suppose $A \cdot ( X + [0,1]^d + \vb) \subseteq K$ for some $A \in \GL_d(\Z)$ and $\vb \in \R^d$.
  We can write $\vb = \vb' - \vb''$ for $\vb' \in \Z^d$ and $\vb'' \in [0,1]^d$.
  Then $A \cdot (X + \vb') \subseteq A \cdot (X + \vb + [0,1]^d) \subseteq K$, i.e., $K$ contains a unimodular copy of $X$.
\end{proof}
\begin{rem}
  In general the inequality in the previous lemma is strict.
  For example, $\flt_1(\{\tfrac{1}{3}\}) = \tfrac{2}{3}$ while $\flt_1^\R(\{\tfrac{1}{3}\} + [0,1]) = 1$.
\end{rem}

Let us have a look how $\flt_d(n \cdot X)$ and $\fltR(X)$ relate as $n$ goes to infinity.
Suppose that $X \subseteq \R^d$ is a full-dimensional convex body, so that there is $\lambda >0$ such that an $\R$-translate of $\lambda^{-1} [0,1]^d$ is contained in $X$.
Then by the previous lemma, we obtain
\[
  n \cdot \fltR(X) = \fltR(n \cdot X) \le \flt_d(n \cdot X) \le \fltR(n \cdot X + [0,1]^d) \le \fltR((n + \lambda) \cdot X) = (n+\lambda) \cdot \fltR(X) \text{.}
\]
We divide this inequality by $n$ and get
\[
  \fltR(X) \le \tfrac{1}{n} \cdot \flt_d(n \cdot X) \le \tfrac{n+\lambda}{n} \cdot \fltR(X)
\]
Hence, in order to determine the order of $\flt_d(n\cdot X)$ as $n\to \infty$, it suffices to determine $\fltR(X)$.

\section{The relation of the lattice width to the Gromov width of symplectic manifolds}
\label{sec:sympl}

\subsection{Background and notation}

The \emph{Gromov width} of a $2d$-dimensional symplectic manifold $(M,\omega)$ is defined as the supremum of the set of capacities $\pi r^2$ of balls of radii $r$ that can be symplectically embedded in $(M,\omega)$ (see \cite{Gromov}). We follow the convention in \cite{McDuff,Lu} and use the identification $S^1=\R/\Z$. There is a large interest in finding lower and upper bounds for the Gromov width, see e.g. \cite{Arezzo, Karshon, Kaveh, Lu, McDuff, PabiniakReduction, PabiniakPolygonal, PabiniakCoadjoint, SchlenkBook}. One should remark that even for symplectic toric manifolds it is not known how to read off the Gromov width from the moment polytope. Here, we observe how closely the Gromov width and the lattice width of the moment polytope are related.

Let $X$ be a complex projective manifold of dimension $d$, $L$ an ample line bundle on $X$ and $\omega$ a K\"ahler form on $X$ representing the Chern class $c_1(L)$. In \cite[Section 4]{Kaveh}, Kaveh constructs $\Z^d$-valued valuations on the field of rational functions $\C(X)$ to obtain associated Newton-Okounkov bodies $\Delta \subseteq \R^d$, i.e., $d$-dimensional convex bodies. We refer to \cite{Kaveh} for the precise definition and choices involved. In the toric case, $\Delta$ is a {\em Delzant polytope}, i.e., a polytope whose normal fan consists of unimodular cones. 

\subsection{Lower bounds on the Gromov width of symplectic manifolds}

The following result (see \cite[Corollary~11.4]{Kaveh}) is closely related to similar results in \cite{McDuff, Lu, PabiniakCoadjoint}.

\begin{thm}[{\cite[Corollary~11.4]{Kaveh}}]
  \label{lowerbd}
  If $\Delta$ contains an $\R$-unimodular copy of $R \cdot \Delta_d$ (for $R > 0$), then the Gromov width of $(X,\omega)$ is at least $R$.
\end{thm}

Note that under the assumption of the theorem, we also have $\width(\Delta) \ge R$. As an immediate application of Corollary~\ref{cor:flatness}, we see that the Gromov width of a symplectic manifold and the lattice width of its moment body are related.

\begin{cor}
  \label{cor:gw}
  The Gromov width of $(X,\omega)$ is bounded from below by $\frac{\width(\Delta)}{\flt(d) d}$.
\end{cor}

In particular, for $d=2$, we get that $0.232 \cdot \width(\Delta)$ is a lower bound on the Gromov width. This bound is surely not sharp. We note that the lower bound in Corollary~\ref{cor:gw} is monotone with respect to inclusion of $\Delta$, a property that is conjectured to hold also for the Gromov width. 

As we will need it later, let us describe a more general lower bound construction which is often used for the Gromov width.

\begin{defn}
  Let $\vb_1, \ldots, \vb_d$ be a lattice basis of $\Z^d$, $\vx \in \R^d$, $a \in \Z_{\ge 1}$, and $k_1,l_1, \ldots, k_d, l_d \in \R_{\ge 0}$ with $k_1 + l_1 = a, \ldots, k_d + l_d = a$. Then
  \[
    \vx + \conv( k_1 \vb_1, - l_1 \vb_1, \ldots, k_d \vb_d, - l_d \vb_d)
  \]
  will be called a {\em diamond of size $a$}. Note that $k_1=\cdots=k_d=1$ and $l_1=\cdots=l_d=0$ yields a unimodular simplex, while for $k_1=l_1=1, \ldots, k_d=l_d=1$ we get the standard crosspolytope. See Figure \ref{fig:diamond} for an illustration.
\end{defn}

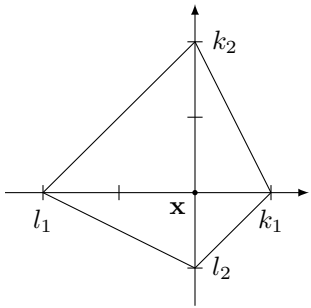
\begin{figure}[!ht]
  \begin{tikzpicture}
    \draw (-2,0) -- (0,2) -- (1,0) -- (0,-1) -- cycle;
    
    \draw[-latex] (-2.5,0) -- (1.5,0);
    \draw[-latex] (0,-1.5) -- (0,2.5);
    
    \fill (0,0) circle (1pt) node[below left] at (0,0) {$\vx$};
    
    \draw (-2,-0.1) node[below] {$l_1$} -- (-2,0.1);
    \draw (1,-0.1) node[below] {$k_1$} -- (1,0.1);
    
    \draw (-0.1,2) -- (0.1,2) node[right] {$k_2$};
    \draw (-0.1,-1) -- (0.1,-1) node[right] {$l_2$};
    
    \draw (-1,-0.1) -- (-1,0.1);
    \draw (-0.1,1) -- (0.1,1);
  \end{tikzpicture}
  \caption{A $2$-dimensional diamond of size $3$ with $\vx = \mathbf{0}$.}
  \label{fig:diamond}
\end{figure}

The following result generalizes Theorem~\ref{lowerbd}. It is strictly speaking only proven in the toric situation, see \cite{SchlenkBook, McDuff, PabiniakPolygonal}, however, the proof of \cite[Corollary~11.4]{Kaveh} should carry through to the general case.

\begin{thm}
  \label{thm:cross}
  If $\Delta$ contains an $\R$-unimodular copy of a diamond of size $R$ (for $R > 0$), then the Gromov width of $(X,\omega)$ is at least $R$.
\end{thm}

Again, in this case we also have $\width(\Delta) \ge R$. We refer to \cite{Bott}, in particular Example 5.4 therein, for the question how sharp such lower bound constructions can be.

\subsection{A conjectural upper bound on the Gromov width of symplectic manifolds}

The only nontrivial upper bound on the Gromov width of a symplectic toric manifold known to the authors was given by Lu (see \cite{Lu}). For this, let us recall the following two definitions from \cite{Lu}.

\begin{defn}
  \label{def-lu}
  Let $\Delta$ be a Delzant polytope with primitive inner facet normals $\vu_k\in (\R^d)^*$ and facets $\{\vx \in \Delta \colon \pro{\vu_k}{\vx} = -\phi_k\}$ for $\phi_k \in \R$. We define two numbers:
  \begin{itemize}
  \item $\Lambda(\Delta)$ is defined as the maximum over positive finite sums of the form $\sum_{k=1}^m a_k \phi_k$, where $a_k$ are nonnegative integers such that $\sum_{k=1}^m a_k \vu_k = \vo$, and $\sum_{k=1}^m a_k \le d+1$.
  \item $\Upsilon (\Delta)$ is defined as the minimum over positive finite sums of the form $\sum_{k=1}^m a_k \phi_k$, where $a_k$ are nonnegative integers such that $\sum_{k=1}^m a_k \vu_k = \vo$.
  \end{itemize}
\end{defn}

\begin{rem}
  Note that $\Lambda(\Delta)$ is a well-defined finite number (cf. \cite{Lu} or \cite[Prop.~3.2]{Bat91}).
  Furthermore, it is straightforward to show that both $\Lambda(\Delta)$ and $\Upsilon(\Delta)$ are invariant under translations by real vectors.
  After an appropriate translation, we may assume that the origin is in the interior of the polytope $\Delta$, and thus all $\phi_k$ are positive.
  From this it straightforwardly follows that a finite sum $\sum_{k=1}^ma_k\phi_k$ as in Definition \ref{def-lu} is positive if and only if at least one $a_k$ is positive.
  Furthermore, note that in general $\Upsilon(\Delta) \le \Lambda(\Delta)$.
\end{rem}

Lu proves the following two results.

\begin{thm}[{\cite[Theorem 1.1]{Lu}}]
  \label{lu-one}
  The Gromov width of $(X,\omega)$ is bounded from above by $\Lambda(\Delta)$.
\end{thm}

\begin{thm}[{\cite[Theorem 1.2]{Lu}, see also \cite[Theorem 5.5]{Bott}}]
  \label{lu-two}
  If $X$ is a toric Fano manifold, then the Gromov width of $(X,\omega)$ is bounded from above by $\Upsilon(\Delta)$.
\end{thm}

We observe that Lu's sharper upper bound $\Upsilon(\Delta)$ is simply the lattice width.

\begin{prop}
  \label{equality-widths}
  If $\Delta$ is a Delzant polytope, then $\Upsilon(\Delta)=\width(\Delta)$.
\end{prop}

The proof follows by combining the next two lemmata. For this, let us define the {\em facet width} of a polytope $\Delta \subseteq \R^d$ as the minimum of $\width_\vu(\Delta)$ where $\vu \in (\Z^d)^*$ ranges over all facet normals of $\Delta$. 

\begin{lem}
  Let $\Delta \subseteq \R^d$ be a Delzant polytope. Then $\Upsilon(\Delta)$ equals the facet width of $\Delta$.\label{lu-upper}
\end{lem}
\begin{proof}
  Let $a_k, \vu_k, \phi_k$ be given as in the definition of $\Upsilon(\Delta)$ (see Definition~\ref{def-lu}).
  
  We first show that $\Upsilon(\Delta)$ is bounded from below by the facet width of $\Delta$. For this, it suffices to show that for any positive finite sum $\sum_{k=1}^m a_k\phi_k$ as in Definition~\ref{def-lu} there is an inner facet normal $\vu$ of $\Delta$ such that $\width_\vu(\Delta) \le \sum_{k=1}^m a_k \phi_k$. Suppose $a_1 \not= 0$ and define $b_1 \coloneqq a_1 - 1$, $b_k \coloneqq a_k$ for $k > 1$. Hence, $\vu_1 + \sum_{k=1}^m b_k \vu_k = \vo$. For $\vx \in \Delta$, we have $\pro{\vu_1}{\vx} \ge - \phi_1$ and $\pro{-\vu_1}{\vx} = \pro{\sum_{k=1}^m b_k \vu_k}{\vx} \ge - \sum_{k=1}^m b_k \phi_k$, so that $\pro{\vu_1}{\vx} \le \sum_{k=1}^m b_k \phi_k$. Thus, $\width_{\vu_1}(\Delta) \le \phi_1 + \sum_{k=1}^m b_k \phi_k = \sum_{k=1}^m a_k \phi_k$.

  For the reverse inequality, suppose $\vu_m \in (\Z^d)^*$ is a primitive inner facet normal of $\Delta$ for which the facet width is attained, i.e., $\width_{\vu_m}(\Delta)=$ facet width of $\Delta$. It suffices to show that there is a positive finite sum $\sum_{k=1}^m a_k \phi_k$ as in Definition~\ref{def-lu} which bounds $\width_{\vu_m}(\Delta)$ from below. Note that both the facet width and $\Upsilon(\Delta)$ are invariant under translations by real vectors. Hence, we may assume that $\vo$ is an element of the facet corresponding to $\vu_m$, ant thus, $\phi_m = 0$. There exist primitive inner facet normals, say $\vu_1, \ldots, \vu_d$ (up to reordering the rays), that span a unimodular cone $\sigma$ of the inner normal fan of $\Delta$ such that $-\vu_m = \sum_{k=1}^d a_k \vu_k$ with $a_k \in \Z_{\ge 0}$. Let the remaining $a_k$ vanish, i.e., $a_{d+1} = \ldots = a_m = 0$, and let $\vx$ be the vertex of $\Delta$ corresponding to $\sigma$. In particular, $\pro{\vu_k}{\vx} = -\phi_k$ for $k=1, \ldots, d$. Then $\sum_{k=1}^m a_k \phi_k = - \sum_{k=1}^d a_k \pro{\vu_k}{\vx} = \pro{\vu_m}{\vx} \le \width_{\vu_m}(\Delta)$.
\end{proof}

It remains to observe the following. 

\begin{lem}
  \label{lw=fw}
  The lattice width of a Delzant polytope coincides with its facet width.
\end{lem}
\begin{proof}
  Let $\Delta \subseteq \R^d$ be a Delzant polytope. Clearly, the lattice width of $\Delta$ is less than or equal to the facet width of $\Delta$.

  For the reverse inequality, let $\vu \in (\Z^d)^*$ with $\width_\vu(\Delta) = \width(\Delta)$. As $\Delta$ is a Delzant polytope, we can replace $\Delta$ by a unimodular copy such that $\vu = \sum_{i=1}^d k_i \ve_i$ for $k_i \in \Z_{\ge 0}$, where the standard basis $\ve_1, \ldots, \ve_d$ of $\Z^d$ spans a cone in the inner normal fan of $\Delta$. Suppose $k_1 \ge 1$ and let $\vv$ be the vertex of $\Delta$ corresponding to the cone spanned by $\ve_1, \ldots, \ve_d$. By translating by a real vector we may assume that $\vv=\vo$. Let $\vw_1$ be a vertex of $\Delta$ maximizing $\ve_1$ on $\Delta$, so that $\width_{\ve_1}(\Delta)=\pro{\ve_1}{\vw_1}$. We have 
  \[
    \pro{\vu}{\vw_1} \le \width_\vu(\Delta) = \width(\Delta) \le \width_{\ve_1}(\Delta) = \pro{\ve_1}{\vw_1}.
  \]
  Hence, $0 \ge \probracket{\vu-\ve_1}{\vw_1} = \probracket{(k_1-1) \ve_1 + \sum_{k=2}^d k_i \ve_i}{\vw_1} \ge 0$, i.e., $\vu(\vw_1) = \ve_1(\vw_1)$. The statement follows.
\end{proof}

This motivates the question whether Lu's sharper upper bound also holds in general.

\begin{conj}
  \label{upper}
  The Gromov width of $(X,\omega)$ is bounded from above by the lattice width of $\Delta$.
\end{conj}

This is also formulated as a question in \cite[Question 5.10]{Bott}. 

\begin{rem}
  \label{heur}
  Here is a heuristical argument in favor of this conjecture. Put a unimodular copy $\Delta'$ of $\Delta$ between two parallel coordinate hyperplanes whose distance equals the lattice width. Then clearly $\Delta'$ is included in a large rectangular box of the same lattice width. It is expected (but far open) that the Gromov width should respect inclusion of the moment polytopes. Therefore, the Gromov width should be at most as large as that of the symplectic toric manifold corresponding to the rectangular box, i.e., a product of projective lines. However, in this case the Gromov width is known (e.g., by Gromov's proof of the non-squeezing theorem) and equals the smallest size of an edge, which is the lattice width of $\Delta$.
\end{rem}

The reader should be aware that the Gromov width may differ from the lattice width as is shown in \cite[Example 5.6]{Bott}. We remark that it seems not to be clear whether fixing the lattice width of $\Delta$ imposes any bounds on the Gromov width.

\begin{rem}
  Lu claims in \cite[Remark 1.5]{Lu} that his upper bound $\Upsilon(\Delta)$ (which by above results equals the lattice width of $\Delta$) does not hold in the non-Fano case by exhibiting an explicit example of a polygon space. This would contradict Conjecture~\ref{upper} and the argument in Remark~\ref{heur} would show that the Gromov width were not monotone with respect to inclusions. However, we couldn't verify his claim, his computations seem to be wrong. In fact, we get $\Upsilon(\Delta) = 2$ while Lu claims $\Upsilon(\Delta)=\tfrac{1}{6}$. In particular, $\Upsilon(\Delta) = 2$ coincides with the Gromov width by \cite[Theorem 1]{PabiniakPolygonal}.
\end{rem}

\begin{rem}
  The reader may wonder why we expressed Conjecture~\ref{upper} in terms of the lattice width instead of the facet width. This was to make the relation to the other results in this paper more apparent and to stress the fact that while the lattice width of convex bodies is monotone with respect to inclusion (as is also conjectured for the Gromov width), this is not true for the facet width of (non-Delzant) polytopes. For instance, for a natural number $k \ge 1$, the lattice triangle with vertices $(0,0)$, $(k,-1)$, and $(k+1,1)$ is contained in the lattice rectangle $[0,k+1] \times [-1,1]$. While the rectangle has facet width $2$, the facet width of the triangle is linearly increasing in $k$ (indeed it equals $2k+1$).
\end{rem}
 
\subsection{\texorpdfstring{Dimension $2$}{Dimension 2}}

As in dimension $2$ any smooth complete toric surface is obtained from $\mathbb{P}^2$, $\mathbb{P}^1 \times \mathbb{P}^1$ or a Hirzebruch surface $\mathcal{H}_a$ (for $a \in \Z_{\ge 1}$) by a sequence of blows-ups at toric fixed points (e.g., see \cite[page 43]{Fulton}), a proof of the $2$-dimensional case of Conjecture \ref{upper} could be achieved by an appropriate generalization of \cite[Theorem 6.2]{Lu}, where Lu studies how the upper bound $\Upsilon(\Delta)$ behaves under blow-ups of toric Fano manifolds at toric fixed points.

Let us note the folklore fact that for these minimal toric surfaces the Gromov width indeed equals the lattice width.

\begin{lem}
  \label{minimal}
  If $X$ is $\mathbb{P}^2$, $\mathbb{P}^1 \times \mathbb{P}^1$ or $\mathcal{H}_a$, then the Gromov width of $(X,\omega)$ equals the lattice width of $\Delta$.
\end{lem}
\begin{proof} 
  The cases $X = \mathbb{P}^2$ and $X = \mathbb{P}^1 \times \mathbb{P}^1$ are straightforwardly verified. It remains to check the case of Hirzebruch surfaces.
  
  We may assume that $\Delta$ is a $4$-gon with vertices $\begin{psmallmatrix}0\\0\end{psmallmatrix}, \begin{psmallmatrix}x\\0\end{psmallmatrix}, \begin{psmallmatrix}0\\y\end{psmallmatrix}, \begin{psmallmatrix}x\\y-ax\end{psmallmatrix}$ with $y > ax > 0$. As $y > x$, the facet width (hence, the lattice width by Lemma~\ref{lw=fw}) of $\Delta$ equals $x$. As $\Delta$ contains $x \Delta_2$, Theorem~\ref{lowerbd} implies that the Gromov width is at least $x$. For the reverse inequality, we distinguish two cases. If $a=1$, then $X$ is Fano, hence, the statement follows by Theorem~\ref{lu-two} and Proposition~\ref{equality-widths} (this was also proven in \cite[p.~206]{SchlenkBook}). If $a >1$, we use Theorem~\ref{lu-one} to bound the Gromov width from above. Indeed, the only nontrivial linear combination (with nonnegative integer coefficients) of ray generators in the inner normal fan of $\Delta$ that sum up to $\begin{psmallmatrix}0\\0\end{psmallmatrix}$ and have at most three summands, is $\begin{psmallmatrix}-1\\0\end{psmallmatrix} + \begin{psmallmatrix}1\\0\end{psmallmatrix} = \begin{psmallmatrix}0\\0\end{psmallmatrix}$, and thus, $\Lambda(\Delta) = x$ by Definition~\ref{def-lu}.
\end{proof}

We couldn't find the following observation on toric surfaces of small Gromov width in the literature.

\begin{prop}
  \label{gw1}
  If $(X,\omega)$ is a symplectic toric surface whose moment polytope $\Delta$ is a Delzant lattice polygon, then its Gromov width equals to $1$ if and only if the lattice width of $\Delta$ is $1$. In this case, $X \cong \mathbb{P}^2$, $X \cong \mathbb{P}^1 \times \mathbb{P}^1$ or $X$ is a Hirzebruch surface.
\end{prop}
\begin{proof}
  If $\Delta$ has no interior lattice points, then it is well known (e.g., see \cite[Theorem 2]{Schicho}) that there are two cases. Either, $\Delta \cong 2 \Delta_2$, so lattice width and Gromov width are both $2$ (e.g. by Theorems~\ref{lowerbd} and \ref{lu-two} together with Proposition \ref{equality-widths}). Or, $\Delta \cong \conv(\begin{psmallmatrix}0\\0\end{psmallmatrix}, \begin{psmallmatrix}1\\0\end{psmallmatrix}, \begin{psmallmatrix}0\\y\end{psmallmatrix}, \begin{psmallmatrix}1\\y-a\end{psmallmatrix})$ with $y \geq a \geq 0$, i.e., $X \cong \mathbb{P}^2$, $X \cong \mathbb{P}^1 \times \mathbb{P}^1$, or $X$ is a Hirzebruch surface (note that we only consider the case where $\Delta$ is a smooth polygon). Hence, the lattice width and the Gromov width both equal $1$ (see Lemma~\ref{minimal}).

  Suppose that $\Delta$ has an interior lattice point. We may assume that $\vo$ is a vertex of $\Delta$ with edge directions $\ve_1, \ve_2$. As $\Delta$ is a Delzant lattice polygon with interior lattice points, convexity implies that $\begin{psmallmatrix}1\\1\end{psmallmatrix} \in \Delta$. Moreover, it is easy to see that $\begin{psmallmatrix}1\\1\end{psmallmatrix}$ cannot be on the boundary, so it is a lattice point in the interior of $\Delta$. Consider the edge with vertex $\vo$ and edge direction $\ve_1$. Let $\vw_1$ be the other vertex on that edge. As $\Delta$ is a Delzant lattice polygon, there exists a lattice point $\vw'_1$ of $\Delta$ that lies in a common edge with $\vw_1$ and has second coordinate $1$. By convexity and as $\begin{psmallmatrix}1\\1\end{psmallmatrix}$ is not on the boundary of $\Delta$, this implies that $(2,1) \in \Delta$. In the same way, one proves that $(1,2) \in \Delta$. This shows that the diamond $\conv(\begin{psmallmatrix}1\\0\end{psmallmatrix},\begin{psmallmatrix}1\\2\end{psmallmatrix},\begin{psmallmatrix}2\\1\end{psmallmatrix},\begin{psmallmatrix}0\\1\end{psmallmatrix})$ of size $2$ is in $\Delta$, and thus the Gromov width is at least $2$ by Theorem~\ref{thm:cross}.
\end{proof}


\section{Proof of Theorem~\ref{thm:flatness}}
\label{sec:flatproof}

The proof uses general results on successive minima and covering minima to reduce the problem to translates of parallelepipeds. 

Let us recall the following standard notions in discrete geometry. The \emph{Minkowski sum} of two subsets $A, B \subseteq \R^d$ is given by $A + B = \{ a + b \with a \in A, b \in B \}$ and it can be recursively extended to finite families $A_1, \ldots, A_k \subseteq \R^d$ in which case we write $\sum_{i=1}^k A_i$. For a convex body $K \subseteq \R^d$, the $d$-th \emph{successive minimum} of its \emph{difference body} $K-K = \{ x - y \with x, y \in K\}$ is defined as follows:
\[
  \lambda_d(K-K) \coloneqq \inf \rleft\{ \lambda>0 \colon \dim \rleft( \lspan \rleft\{ \lambda (K-K) \cap \Z^d \rright\} \rright) = d \rright\} \text{,}
\]
while the definition of the $d$-th \emph{covering minimum} of $K$ is given by:
\[
  \mu_d(K) \coloneqq \inf \rleft\{  \mu > 0 \colon \mu K + \Z^d = \R^d \rright\}.
\]

\begin{lem}
  \label{lem:KL}
  Let $K \subseteq \R^d$ be a convex body. Then we have $\lambda_d ( K - K ) \le \tfrac{\flt(d)}{\width(K)}$.
\end{lem}
\begin{proof}
  We set $\tau \coloneqq \tfrac{\flt(d)}{\width(K)}$. By \cite[Lemma 2.4]{KL:CovMin}, $\lambda_d(K - K) \le \mu_d(K)$. It suffices to show $\mu_d(K) \le \tau$. For this, let $\vect{x} \in \R^d$ and set $K' \coloneqq \tau K - \vect{x}$. As $\width\rleft( K' \rright) = \flt(d)$, there is a lattice point $\vect{y}$ in $K'$, i.e.~$\vect{y} = \vect{z} - \vect{x}$ for some $\vect{z} \in \tau K$. Thus $\vect{x} \in \tau K + \Z^d$.
\end{proof}

Given two points $\vv, \vw \in \R^d$, we write $\rleft[ \vv, \vw \rright]$ for the line segment between those two points, i.e., $\rleft[ \vv, \vw \rright] = \conv \rleft( \vv, \vw \rright)$. 

Let us recall the following folklore result. For the reader's convenience, we include a proof.

\begin{lem}
  \label{lem:aff-basis}
  Let $\vv_1, \ldots, \vv_d \in \Z^d$ be linearly independent. If $\va \in \Z^d$ (resp.~$\va \in \R^d$), then the lattice parallelepiped $P \coloneqq \va + \sum_{i=1}^d \rleft[ \vect{0}, \vv_i \rright]$ (resp.~the parallelepiped $P \coloneqq \va + \sum_{i=1}^d \rleft[ \vect{0}, 2 \vv_i \rright]$) contains a unimodular copy of the standard simplex $\Delta_d$.
\end{lem}
Note the stronger assumptions in the case of real translates of lattice parallelepipeds.
\begin{proof}
  To simplify notation set $\tau = 1$ or $\tau = 2$ depending on whether $\va \in \Z^d$ or $\va \in \R^d$ respectively. The basic idea of the proof is independent of the case which we are considering: We proceed by induction on the dimension $d$. The crucial difference in the argument is evident in dimension $d=1$: If $\va \in \Z^d$, then it suffices to observe that any interval with integral endpoints contains two consecutive integer points. However, if $\va \in \R^d$, then it is enough to assume that $P$ is a segment of length at least $2$ to ensure that $P$ contains two consecutive integer points.
    
  Now, let $d \ge 2$. By applying an appropriate (linear) unimodular transformation, we may assume $\vv_1, \ldots, \vv_{d-1} \in \Z^{d-1} \times \rleft\{ \vect{0} \rright\}$. Consider the projection $\pi \colon \R^d \to \R$ onto the last coordinate, i.e., given by $\pi \rleft( x_1, \ldots, x_d \rright) = x_d$. We have $\pi(P) = \pi(\va) + \rleft[ \vect{0}, \tau \cdot \pi\rleft( \vv_d \rright) \rright]$ with $0 \neq \pi \rleft( \vv_d \rright) \in \Z$. As $\pi(P)$ is a segment, it follows by the base case of the induction that it contains two consecutive integers, say $k$ and $k+1$. Thus (as $\vv_1, \ldots, \vv_{d-1}$ have last coordinate $0$) there are $\vv',\vv'' \in \rleft[ \vect{0}, \tau\cdot \vv_d \rright]$ such that $\pi \rleft( \va + \vv' \rright) = k$ and $\pi \rleft( \va + \vv'' \rright) = k+1$. We note that $\va + \vv'$, $\va + \vv''$ need not to be lattice points.
  
  We apply the induction hypothesis to $P' \coloneqq \va + \vv' + \sum_{i=1}^{d-1} \rleft[ \vect{0}, \tau \cdot \vv_i \rright] \subseteq \R^{d-1} \times \{k\}$ and deduce the existence of a unimodular copy $\vp'_0,\ldots, \vp'_{d-1} \in P'$ of $\vect{0}, \vect{e}_1, \ldots, \vect{e}_{d-1}$. Analogously, the induction hypothesis yields that the set $P'' \coloneqq \va + \vv'' + \sum_{i=1}^{d-1} \rleft[ \vect{0}, \tau\cdot \vv_i \rright] \subseteq \R^{d-1} \times \{k+1\}$ also contains a unimodular copy $\vp''_0, \ldots, \vp''_{d-1}$ of $\vect{0}, \ve_1, \ldots, \ve_{d-1}$. By construction, for any $i \in \{0, \ldots, d-1\}$ the vectors $\vp'_0, \ldots, \vp'_{d-1}, \vp''_i$ form a unimodular copy of $\vect{0}, \ve_1, \ldots, \ve_d$ contained in $P$.
\end{proof}

\begin{rem}
  Clearly, the previous proof implies more than what we need. For instance, it can be easily modified to show that any such parallelepiped $P$ has at least $2^d$ many lattice points. However, the reader is cautioned not to jump to the conclusion that it proves the existence of a unimodular copy of the unit cube $[0,1]^d$ in $P$.
\end{rem}

\begin{proof}[Proof of Theorem~\ref{thm:flatness}]
  (1)  By Lemma \ref{lem:KL}, $\lambda_d(K - K) \le \frac{1}{2d}$, i.e., there exist $d$ segments $I_i = [\va_i,\vb_i]$ with $i \in \{1,\ldots,d\}$ contained in $\frac{1}{2d} K$ such that the $d$ vectors $\vv_i \coloneqq \vb_i - \va_i$ are linearly independent and belong to $\Z^d$. This implies $I_1 + \cdots + I_d \subseteq \frac{1}{2} K$. Hence,
   \[
   2\rleft( I_1 + \cdots + I_d \rright) \subseteq K \text{.}
  \]
  The latter inclusion can be written as
  \[
    \va + \sum_{i=1}^d \rleft[ \vect{0}, 2 \vv_i \rright] \subseteq K
  \]
  with $\va \coloneqq 2 \rleft( \va_1 + \ldots + \va_d \rright)$. The statement follows from Lemma \ref{lem:aff-basis}.

  (2) Again by Lemma \ref{lem:KL}, $\lambda_d(K - K) \le \frac{1}{d}$, i.e., there exist $d$ segments $I_i = [\va_i,\vb_i]$ with $i \in \{1,\ldots,d\}$ contained in $\frac{1}{d} K$ such that the $d$ vectors $\vv_i \coloneqq \vb_i - \va_i$ are linearly independent and belong to $\Z^d$. This implies $I_1 + \cdots + I_d \subseteq K$. Hence, 
  \[
   \va + \sum_{i=1}^d \rleft[ \vect{0}, \vv_i \rright] \subseteq K
  \]
  with $\va \coloneqq \va_1 + \ldots + \va_d$. The statement follows from Lemma \ref{lem:aff-basis}.
\end{proof}


\section{Proof of Proposition~\ref{prop:example}}
\label{sec:exampleproof}

For given $k \in \N$ with $k \ge 3$, we define the three-dimensional lattice polytope
\[
  P \coloneqq \conv( \{(3,0,-1),(0,2,-1)\} \cup [0,k]^3) \qquad \text{(see Figure \ref{fig:3d_polytope})}.
\]
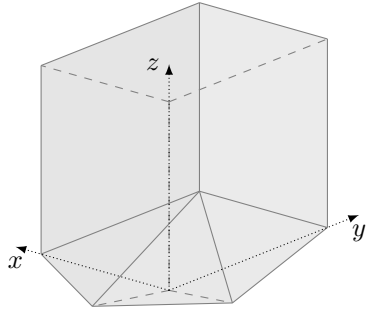
\begin{figure}[!ht]
  \begin{tikzpicture}[rotate around x = 180, rotate around y = -60,scale=0.5]
  \def\k{5}
    \fill[opacity=0.2,fill=gray] (0,0,\k) -- (\k,0,\k) -- (\k,\k,\k) -- (0,\k,\k) -- cycle;
    \fill[opacity=0.2,fill=gray] (0,\k,\k) -- (\k,\k,\k) -- (0,\k+1,3) -- cycle;
    \fill[opacity=0.2,fill=gray!75] (\k,0,0) -- (\k,0,\k) -- (\k,\k,\k) -- (\k,\k,0) -- cycle;  
    \fill[opacity=0.2,fill=gray!90] (\k,\k,\k) -- (0,\k+1,3) -- (2,\k+1,0) -- cycle;
    \fill[opacity=0.2,fill=gray!80] (\k,\k,\k) -- (2,\k+1,0) -- (\k,\k,0) -- cycle;

    \draw[gray,thin,dashed] (0,0,0) -- (\k,0,0);
    \draw[gray,thin,dashed] (0,0,0) -- (0,\k,0);
    \draw[gray,thin,dashed] (0,0,0) -- (0,0,\k);
    \draw[gray,thin,dashed] (0,\k,0) -- (0,\k+1,3);
    \draw[gray,thin,dashed] (0,\k,0) -- (2,\k+1,0);
    \draw[gray] (0,\k+1,3) -- (\k,\k,\k);
    \draw[gray] (0,\k+1,3) -- (2,\k+1,0);
    \draw[gray] (2,\k+1,0) -- (\k,\k,0);
    \draw[gray] (2,\k+1,0) -- (\k,\k,\k);
    \draw[gray] (\k,\k,\k) -- (\k,\k,0);
    \draw[gray] (\k,\k,\k) -- (\k,0,\k);
    \draw[gray] (\k,0,0) -- (\k,0,\k);
    \draw[gray] (0,0,\k) -- (\k,0,\k);
    \draw[gray] (\k,0,0) -- (\k,\k,0);
    \draw[gray] (0,\k,\k) -- (\k,\k,\k);
    \draw[gray] (0,\k,\k) -- (0, \k+1,3);
    \draw[gray] (0,0,\k) -- (0,\k,\k);

    \draw[-latex,thin,densely dotted] (0,\k,0) -- (\k+1,\k,0) node[below] {$y$};
    \draw[-latex,thin,densely dotted] (0,\k,0) -- (0,-1,0) node[left] {$z$};
    \draw[-latex,thin,densely dotted] (0,\k,0) -- (0,\k,\k+1) node[below] {$x$};
  \end{tikzpicture}
  \caption{The polytope $P$ for $k=5$.}
  \label{fig:3d_polytope}
\end{figure}

We will show that this polytope $P$ of width $k$ has the desired property. Then for $d > 3$, one simply takes the Cartesian product of $P$ with $[0,k]^{d-3}$. 

Let $t \in \N$ with $t \ge 2$. Let us abbreviate $Z_1 \coloneqq \Z^3 \cap P$, and 
\[
  Z_t \coloneqq \underbrace{Z_1 + \cdots + Z_1}_{t \; \text{times}}.
\]
We will show that the lattice point
\[
  p \coloneqq (3t-4,1,1-t)
\]
satisfies $p \in \conv(Z_t) = t P$ but $p \not\in Z_t$. 

For this, we consider
\[
  Z'_t \coloneqq \setcond{\vx \in \Z^2}{(\vx,1-t) \in Z_t}.
\] 
In order to determine $Z'_t$, we observe that the last coordinates of the points in $Z_1$ comprise the set $\{-1,0,\ldots,k\}$. We thus need to determine the possibilities for values $s_1,\ldots,s_t \in \{-1,0,\ldots,k\}$ such that the sum $s_1 + \cdots + s_t$ is equal to $1-t$. It is easily seen that one of the values $s_1,\ldots,s_t$ has to be equal to $0$ and the remaining ones to $-1$ for the latter to be fulfilled. Since $(3,0,-1)$ and $(0,2,-1)$ are the only two points in $Z_1$ with last component $-1$ and $\{0,\ldots,k\}^2 \times \{0\}$ is the set of all points of $Z_1$ with last component $0$, we get
\begin{equation}
  \label{eq:L}
  Z'_t  = \underbrace{\{(3,0),(0,2)\} + \cdots + \{(3,0),(0,2)\}}_{(t-1) \; \text{times}} + \{0,\ldots,k\}^2 \text{.}
\end{equation}

We observe that
\[
  (3t-4,1) = \frac{1}{2} \left((t-1)(3,0)+(0,0)\right)+\frac{1}{2}\left((t-2)(3,0)+(0,2)+(1,0)\right) \in \conv(Z'_t) \text{.}
\]
Thus, $p \in \conv(Z_t)$. 

Finally, we notice that $(3t-4,1)=(3(t-1)-1,1) \not\in Z'_t$. Otherwise, it could be expressed as a sum as on the right side of \eqref{eq:L}. However, as the second component of $\{0,\ldots,k\}^2$ is nonnegative, no $(0,2)$ would be allowed as a summand, so it would be a sum of $(t-1) (3,0)$ and a point in $\{0,\ldots,k\}^2$, a contradiction. Therefore, $p \not\in Z_t$. \hfill$\qed$

\newcommand{\etalchar}[1]{$^{#1}$}
\providecommand{\bysame}{\leavevmode\hbox to3em{\hrulefill}\thinspace}

\end{document}